\documentclass[12pt,oneside]{amsart}

\usepackage{amssymb}
\usepackage{amsmath}
\usepackage{verbatim}
\usepackage[retainorgcmds]{IEEEtrantools}

\usepackage[margin=30 mm,heightrounded=true,centering]{geometry}
\theoremstyle{plain}
 \newtheorem{theorem}{Theorem}[section]
 \newtheorem{lemma}[theorem]{Lemma}
 \newtheorem{corollary}[theorem]{Corollary}
      
 \theoremstyle{definition}
 \newtheorem{definition}[theorem]{Definition}
      
 \theoremstyle{remark}

\makeatletter
\def\@setcopyright{}
\def\serieslogo@{}
\makeatother

\begin{document}

\author[Ran Ji]{RAN JI}

\title[Asymptotic Dirichlet Problem]{The asymptotic Dirichlet problems on manifolds with unbounded negative curvature}

\maketitle

\begin{abstract}
Elton P. Hsu used probabilistic method to show that the asymptotic Dirichlet problem is uniquely solvable under the curvature condition $-C e^{2-\eta}r(x) \leq K_M(x)\leq -1$ with $\eta>0$. We give an analytical proof of the same statement. In addition, using this new approach we are able to establish two boundary Harnack inequalities under the curvature condition $-C e^{(2/3-\eta)r(x)} \leq K_M(x)\leq -1$ with $\eta>0$. This implies that there is a natural homeomorphism between the Martin boundary and the geometric boundary of $M$. As far as we know, this is the first result of this kind under unbounded curvature conditions. Our proof is a modification of an argument due to M. T. Anderson and R. Schoen. 
\end{abstract}

\maketitle

\section{Introduction}

In this paper we discuss the solvability of the asymptotic Dirichlet problem and the equivalence of the geometric and Martin boundary on manifolds with negative curvature. 

Let $M$ be a complete, simply connected n-dimensional Riemannian manifold whose sectional curvature is bounded from above by a negative constant. Fix a base point $p \in M$. It is well known that the exponential map $\mathrm{exp}_p: \mathrm{T}_p M \to M$ is a diffeomorphism. $S(\infty)$, which is defined as the set of equivalence classes of geodesic rays, can be identified with the unit sphere in $\mathrm{T}_p(M)$. A basic fact is that $\overline{M}=M \cup S(\infty)$ with the `cone topology' is a compactification of $M$ \cite{MR1333601}.

Given $\varphi \in C^0(S(\infty))$, the asymptotic Dirichlet problem is to find a continuous function $f$ on $\overline{M}$ such that $f$ is harmonic on $M$ and $f=\varphi$ on $S(\infty)$. The case when $M$ has pinched curvature was solved in 1983 independently by Anderson \cite{And83} and Sullivan \cite{MR730924}. Anderson's approach was to construct appropriate convex sets and use the convexity property of Choi \cite{MR722769}. A simpler proof was given by Anderson and Schoen \cite{MR794369} in 1985. In 1992, Borb\'ely was able to replace the lower bound of the curvature by an unbounded growth function. His proof was based upon that of Anderson, namely he proved the following theorem.

\begin{theorem}(\cite{MR1069289})
\label{Borbely}
Let $M$ be a complete, simply connected Riemannian manifold with negative sectional curvature. Let $r=d(p,\cdot )$ denote the distance function and $\lambda<\dfrac{1}{3}$ be a positive constant. If the sectional curvature $K_M$ satisfies  $K_M(x)\leq -1$ everywhere and $- e^{\lambda r(x)} \leq K_M(x)$ outside a compact subset of $M$, then the asymptotic Dirichlet problem is uniquely solvable.
\end{theorem}

Hsu was able to get a better lower bound of the curvature condition using probabilistic method. His result is as follows.

\begin{theorem} (\cite{MR1988474})
\label{main}
Let $M$ be a complete, simply connected Riemannian manifold whose sectional curvature $K_M$ satisfies $- Ce^{\lambda r(x)} \leq K_M(x)\leq -1$ on $M$ for some $\lambda<2$. Then the asymptotic Dirichlet problem is uniquely solvable.
\end{theorem}

We will give an analytical proof of Theorem \ref{main} in Section 3 based upon that of Anderson and Schoen \cite{MR794369}. A key refinement is that  instead of taking the average $\bar{\varphi}$ of the extended function $\varphi$ in a ball of fixed radius, we let the radius vary. Then with the help of Bishop volume comparison theorem, we can show that even under relaxed curvature growth condition, the argument still works and yields Hsu's result.

On a non-parabolic manifold, i.e., a manifold possesses positive Green's function, one can define the Martin boundary which describes the behavior of harmonic functions at infinity. We will give more details in section 4. A natural question is whether the Martin boundary is the same as the geometric boundary. Anderson and Schoen showed that we can identify them when the manifold has pinched negative curvature.
\begin{theorem}(\cite{MR794369})
\label{AS}
Let $M$ be a complete, simply connected Riemannian manifold whose sectional curvature satisfies $-b^2 \leq K_M \leq -a^2<0$. Then there exists a natural homeomorphism $\Phi:\mathcal{M} \to S(\infty)$ from the the Martin boundary $\mathcal{M}$ of $M$ to the geometric boundary $S(\infty)$. Moreover, $\Phi^{-1}$ is H\"{o}lder continuous.
\end{theorem}

To prove Theorem \ref{AS}, they established two boundary Harnack inequalities, which estimate the growth of positive harmonic functions in cones which vanish continuously at infinity. In Section 5, we relax the curvature assumption in Theorem \ref{AS} and establish the Harnack inequalities. It follows that the Martin boundary can be identified with the geometric boundary. To be precise, we prove the following theorem.
\begin{theorem}
\label{Martin}
Let $M$ be a complete, simply connected Riemannian manifold whose sectional curvature $K_M$ satisfies $- Ce^{\lambda r(x)} \leq K_M(x)\leq -1$ on $M$ for some $\lambda<\dfrac{2}{3}$. Then there is a natural homeomorphism between the geometric boundary and Martin boundary of $M$.
\end{theorem}
Our result on the Martin boundary is the first one that allows the sectional curvature go to $-\infty$ as $r \to \infty$. 

\emph{Remark 1.1} While Theorem \ref{main} holds for $\lambda<2$, our proof of Theorem \ref{Martin} is only valid for $\lambda<\dfrac{2}{3}$. This is technically from applying the boundary Harnack inequalities, which are proved only under the stronger curvature condition in our paper. It is possible that those Harnack inequalities are true under relaxed curvature condition.
\\

In the case when the upper bound of the curvature approaches to $0$ at infinity we have similar results. For such manifolds, Hsu used probabilistic method to prove the following theorem on the asymptotic Dirichlet problem.
\begin{theorem}(\cite{MR1988474})
\label{main2}
Let $M$ be a complete, simply connected Riemannian manifold. If the sectional curvature satisfies $K_M \leq -\dfrac{\alpha(\alpha-1)}{r^2}$ for some $\alpha>2$ and the Ricci curvature satisfies $Ric_M \geq -r^{2\beta}$ for some $\beta<\alpha-2$, then the asymptotic Dirichlet problem is uniquely solvable. 
\end{theorem}

An analytic proof of Theorem \ref{main2} will be given in Section 6.  In addition, it is showed in section 7 that the Martin boundary can be identified with the geometric boundary under a stronger curvature condition.
\begin{theorem}
\label{Martin2}
Let $M$ be a complete, simply connected Riemannian manifold. If the sectional curvature satisfies $K_M \leq -\dfrac{\alpha(\alpha-1)}{r^2}$ for some $\alpha>2$ and the Ricci curvature satisfies $Ric_M \geq -r^{2\beta}$ for some $\beta<\dfrac{\alpha-4}{3}$. Then there is a natural homeomorphism between the geometric boundary and Martin boundary of $M$.
\end{theorem}

\textbf{Acknowledgements} The author would like to thank Professor J\'ozef Dodziuk for the invaluable support and guidance. The author would also like to thank Professor Zheng Huang and Professor Marcello Lucia for their useful comments and suggestions.

\section{Preliminaries}

Throughout this section we assume $M$ is a complete, simply connected Riemannian manifold of n dimensions with sectional curvature $K_M(x)\leq -1$.

Denote by $H(-1)$ the two-dimensional hyperbolic plane with constant curvature $-1$. We have the following well known Toponogov comparison theorem \cite{MR1333601}.

\begin{theorem}
\label{Toponogov}
Let $\triangle pxy$ be a geodesic triangle in $M$ with vertices $p,x,y$. Suppose $\triangle\tilde{p}\tilde{x}\tilde{y}$ is the corresponding geodesic triangle in $H(-1)$, such that the corresponding sides have the same length. Then we have $$\angle(px,py)\leq\angle(\tilde{p}\tilde{x},\tilde{p}\tilde{y}),$$
where $\angle(px,py)$ denotes the angle at $p$ between the geodesic segments $px$ and $py$.
\end{theorem}

In this proof we assume that all geodesics are parameterized by arc length. 

Two geodesic rays $\gamma_1$ and $\gamma_2$  are said to be equivalent, denoted by $\gamma_1 \sim \gamma_2$ if there exists a constant $C$ such that for any $t \geq 0$ we have
$$d(\gamma_1(t),\gamma_2(t)) \leq C.$$

Define $S(\infty)$, the sphere at infinity, to be
$$S(\infty)= \text{the set of all geodesic rays}/ \sim.$$ 

Let $S_p$ denote the unit sphere in $T_p(M)$. Given $\omega \in S_p$, there exists a unique geodesic ray $\gamma:[0,\infty) \to M$ satisfying $\gamma(0)=p$ and $\gamma'(0)=\omega$. Two geodesic rays $\gamma_1$ and $\gamma_2$ starting from $p$ are equivalent  if and only if $\gamma_1=\gamma_2$. At the same time each equivalence class contains a representative emanating from $p$. Thus $S(\infty)$ can be identified with $S_p$ for each $p \in M$.

Now we can define the cone $C_p(\omega, \delta)$ around $\omega$ of angle $\delta$ by
$$C_p(\omega, \delta)=\{x \in M:\angle(\omega,\gamma'_{px}(0))<\delta\},$$
where $\gamma_{px}$ denotes the geodesic ray starting from $p$ that passes through $x$. We call
$$T_p(\omega,\delta,R)=C_p(\omega, \delta) \setminus B_p(R)$$
a truncated cone of radius $R$. We denote $M \cup S(\infty)$ by $\overline{M}$. Then the set of $T_p(\omega,\delta,R)$ for all $\omega \in S_p$, $\delta$ and $R>0$ and $B_q(r)$ for all $q \in M$ and $r>0$ form a basis of a topology on $\overline{M}$, which is called the cone topology. This topology makes $\overline{M}$ a compactification of $M$ \cite{MR1333601}.

\emph{Remark 2.1} The cone topology on $\overline{M}$ is independent of the choice of $p$.

\emph{Remark 2.2} Anderson and Schoen showed that if $-b^2 \leq K_M \leq -a^2<0$, then the topological structure is $C^\alpha$, where $\alpha=a/b$.\\

From now on we identify $S(\infty)$ with $S_p$ and its image under the exponential map $\mathrm{exp}_p(S_p)$. Let $(r,\theta)$ be the normal polar coordinates at $p$. Then $\varphi \in C^0(S(\infty))$ can be written as $\varphi=\varphi(\theta)$. Assume that Theorem \ref{main} is true for all $\varphi \in C^\infty(S_p)$. Given $\varphi \in C^0(S_p)$, let $\varphi_n \in  C^\infty(S_p)$ be a sequence of functions such that $\varphi_n \to \varphi$ uniformly. Then there exists a sequence of harmonic functions $u_n \in C^\infty(M) \cap C^0(\overline{M})$ satisfying $u_n(r,\theta)\to\varphi_n(\theta)$ as $r \to \infty$. By the maximum principle $u_n \to u$ uniformly on $\overline{M}$ and $u|_{S(\infty)}=\varphi$. This shows that without loss of generality, we may assume $\varphi \in C^\infty(S_p)$.

Extend $\varphi$ to $M \setminus \{p\}$ by defining
$$\varphi(r,\theta)=\varphi(\theta)$$
for $r>0$. We still use the letter $\varphi$ to denote the extended function. Then $\varphi $ is smooth and bounded on $M \setminus \{p\}$.

Let
$$
\mathrm{osc}_{B_x(d)} \varphi=\sup_{y \in B_x(d)} |\varphi(y)-\varphi(x)|$$
be the oscillation of $\varphi$ in the geodesic ball $B_x(d)$.

Since $\varphi \in C^\infty(S_p)$, it is Lipschitz continuous on $S_p$. We have for $y \in B_x(d)$,
\begin{equation}
|\varphi(y)-\varphi(x)|=|\varphi(\theta')-\varphi(\theta)| \leq C|\theta'-\theta| = C \angle({px},{py}), \label{1}
\end{equation}
where $\theta$, $\theta'$ are the spherical coordinates of $x$ and $y$ respectively.

Now it is necessary to estimate the  angle $\angle({px},{py})$.

\begin{lemma}
\label{hyperbolic}
Let $M$ be a complete, simply connected Riemannian manifold with sectional curvature $K_M \leq -1$, and let $p,x,y$ be three points in $M$. Suppose that $d(p,x)=s$, and $y \in B_x(d)$ with $d<s$. We have
$$\angle(px,py)<\frac{2d}{e^{s-d}-1}.$$
\end{lemma}

The proof is based on a computation in the hyperbolic plane and the Topogonov comparison theorem. This lemma is similar to that in \cite{MR1069289}. For completeness, we include the proof here.

Let $\triangle\tilde{p}\tilde{x}\tilde{y}$ be the corresponding geodesic triangle in $H(-1)$ such that $d(\tilde{p},\tilde{x})=d(p,x)=s$, $d(\tilde{x},\tilde{y})=d(x,y)=d' < d<s$ and $d(\tilde{p},\tilde{y})=d(p,y)$.  We use the Poincare disk model to compute $\angle(\tilde{p}\tilde{x},\tilde{p}\tilde{y})$ in the unit Euclidean ball $B^2$ with metric
\begin{equation}
\label{metric}
\mathrm{d}s^2_H=4 \frac{dr^2+r^2d\phi^2}{(1-r^2)^2},
\end{equation}
where $(r,\phi)$ are the polar coordinates of $B^2$.

Without loss of generality, we may assume that $\tilde{p}$ is the center of $B^n$. Let $\tilde{\tilde{x}}$ be the intersection of  the geodesic sphere $S_{\tilde{x}}(d)$ and the line segment $\tilde{p}\tilde{x}$. Then $d_H(\tilde{p},\tilde{\tilde{x}})=s-d$. From  \eqref{metric} we can easily compute the Euclidean distance between $\tilde{p}$ and $\tilde{x}$:
\begin{equation*}
d_E(\tilde{p},\tilde{\tilde{x}})=\frac{e^{s-d}-1}{e^{s-d}+1}.
\end{equation*}

Let $\tilde{\tilde{y}}$ be the intersection of the geodesic sphere $S_{\tilde{p}}(s-d)$ and the line segment $\tilde{p}\tilde{y}$. Denote by $\mathrm{arc}(\tilde{\tilde{x}},\tilde{\tilde{y}})$ the circular arc joining $\tilde{\tilde{x}}$ and $\tilde{\tilde{y}}$, $l_E$ and $l_H$ the lengths of curves in Euclidean and hyperbolic metrics respectively. We have $l_H(\mathrm{arc}(\tilde{\tilde{x}},\tilde{\tilde{y}}))<d_H(\tilde{x},\tilde{y})=d'$. In fact, let $\gamma_1(\phi)=(d_E(\tilde{p},\tilde{\tilde{x}}),\phi)$ and $\gamma_2(\phi)=(r(\phi),\phi)$ be the parameterization of $\mathrm{arc}(\tilde{\tilde{x}},\tilde{\tilde{y}})$ and the geodesic segment $\tilde{x}\tilde{y}$ respectively. We have
$$|\gamma'_1(\phi)|= \frac{2 d_E(\tilde{p},\tilde{\tilde{x}})}{1-(d_E(\tilde{p},\tilde{\tilde{x}}))^2}.$$

Also
$$|\gamma'_2(\phi)|=2 \frac{\sqrt{r'^2(\phi)+r^2(\phi)}}{1-r^2(\phi)} \geq \frac{2r(\phi)}{1-r^2(\phi)}.$$

We have $r(\phi) > d_E(\tilde{p},\tilde{\tilde{x}})$ for all $\phi$ since the geodesic ball $B_{\tilde{x}}(d')$ lies completely outside $B_{\tilde{p}}(s-d)$, which implies $|\gamma'_2(\phi)| > |\gamma'_1(\phi)|$ and thus $l_H(\mathrm{arc}(\tilde{\tilde{x}},\tilde{\tilde{y}})) < d(\tilde{x},\tilde{y})=d'<d$. By \eqref{metric} again we have
\begin{eqnarray*}
l_E(\mathrm{arc}(\tilde{\tilde{x}},\tilde{\tilde{y}})) &\leq& \frac{1}{2} \cdot (1-(d_E(\tilde{p},\tilde{\tilde{x}}))^2) \cdot l_H(\mathrm{arc}(\tilde{\tilde{x}},\tilde{\tilde{y}})) \\
&<& \frac{d}{2} \cdot (1-(d_E(\tilde{p},\tilde{\tilde{x}}))^2).
\end{eqnarray*}

Then 
\begin{eqnarray*}
\angle(\tilde{p}\tilde{x},\tilde{p}\tilde{y})=\angle(\tilde{p}\tilde{\tilde{x}},\tilde{p}\tilde{\tilde{y}}) &=& \frac{l_E(\mathrm{arc}(\tilde{\tilde{x}},\tilde{\tilde{y}}))}{d_E(\tilde{p},\tilde{\tilde{x}})} \\
&<& \frac{d}{2} \cdot \frac{1-(d_E(\tilde{p},\tilde{\tilde{x}}))^2}{d_E(\tilde{p},\tilde{\tilde{x}})} \\
&<& \frac{2d}{e^{s-d}-1}.
\end{eqnarray*}

By Theorem \ref{Toponogov} we have $\angle(px,py)\leq\angle(\tilde{p}\tilde{x},\tilde{p}\tilde{y})<\dfrac{2d}{e^{s-d}-1}$. Lemma \ref{hyperbolic} is proved.\\

\section{Proof of Theorem \ref{main}}

Throughout this section we assume $M$ is a complete, simply connected n-dimensional Riemannian manifold with sectional curvature bounded from above by $-1$ and satisfies
$$-Ce^{(2-2\delta)r(x)} \leq K_M(x)$$
outside a compact subset of $M$ for some $\delta>0$.

\emph{Remark 1} Without loss of generality, we may assume $-Ce^{(2-2\delta)r(x)} \leq K_M(x) \leq -1$ for some large enough constant $C$ on the whole manifold.

\emph{Remark 2}  The factor $2$ before $\delta$ is just for notational convenience.\\

We follow Anderson and Schoen's argument. 

Let $$d(x) =e^{-(1-\delta)r(x)}.$$

We estimate the oscillation of $\varphi$ in the geodesic ball $B_x(d(x))$.
Combining equation \eqref{1} and Lemma \ref{hyperbolic} we see easily that 
\begin{equation}
\label{osc}
\mathrm{osc}_{B_x(d(x))} \varphi=\mathrm{O}(e^{-(2-\delta)r(x))}).
\end{equation}

Now we take the average $\bar{\varphi}$ of $\varphi$ in the ball $B_x(d(x))$ in the following way. Let $\chi \in C_0^\infty(\mathbb{R})$ be a function satisfying $0\leq \chi \leq 1$, $\chi(t)=0$ for $|t|\geq1$ and $\chi(t)=1$ for $|t| \leq 1/4$. Let $$u(x,y)=\chi(e^{(2(1-\delta)r(x)} \rho_x^2(y)),$$ where $\rho_x=d(x,\cdot)$. We have
\begin{equation}
\label{u}
u(x,y)=
\begin{cases}
1 & \text{if } y \in\overline{ B_x(d(x)/2)}\\
0 & \text{if } y \in M \setminus B_x(d(x)).
\end{cases}
\end{equation}

Now define
$$\overline{\varphi}(x)=\dfrac{\int_M u(x,y) \varphi(y) dy}{\int_M u(x,y) dy}.$$

Since $\varphi$ is continuous and bounded on $M \setminus\{p\}$, $\overline{\varphi}$ is smooth on $M$. Then we have 
\begin{eqnarray*}
|\overline{\varphi}(x)-\varphi(x)| &=& \frac{\int_{B_x(d(x))} u(x,y) (\varphi(y)-\varphi(x)) dy}{\int_{B_x(d(x))} u(x,y) dy} \\
&\leq &\sup_{y \in B_x(d(x))} |\varphi(y)-\varphi(x)| \\
&= & \mathrm{osc}_{B_x(d(x))} \varphi \\
& =& \mathrm{O}(e^{-(2-\delta)r(x))}),
\end{eqnarray*}
which implies $\overline{\varphi}$ and $\varphi$ have the same value on $S(\infty)$.

Let $$v(x)=\int_M u(x,y) dy,$$ 
it follows from \eqref{u} that $\mathrm{Vol}(B_x(d(x)/2) )\leq v(x) \leq \mathrm{Vol}\left(B_x\left(d\left(x\right)\right)\right)$. 

In the following we will simply write $d$ for $d(x)$, $u$ for $u(x,y)$, $\rho$ for $\rho_x(y)$ and $v$ for $v(x)$ and the operations $\nabla$ and $\Delta$ will always be with respect to $x$. We have
\begin{eqnarray}
\Delta \overline{\varphi}(x_0) &=& \Delta\left(\overline{\varphi}\left(x\right)-{\varphi}\left(x_0\right)\right) |_{x=x_0} \\
&=&\nonumber  \int_M \Delta( \frac{u}{v})(\varphi(y)-\varphi(x_0)) dy  |_{x=x_0}.
\end{eqnarray}

Direct computation gives
\begin{equation}
\label{u/v}
\Delta\left(\frac{u}{v}\right)=\frac{v \Delta u-2\nabla u \cdot \nabla v -u \Delta v}{v^2}+\frac{2u}{v^3}|\nabla v|^2.
\end{equation}

Since $r$ and $\rho$ are both distance functions, we have $|\nabla r|=|\nabla \rho|=1$. Together with the fact that $\mathrm{supp} \,u \subset \overline{B_x(d(x))}$, we have

\begin{eqnarray}
\label{nabla u}
\nabla u &=& \chi'(e^{(2(1-\delta)r(x)}\rho_x^2(y)) \cdot \left( e^{2\left(1-\delta\right)r}\left(\left(2-2\delta\right)\rho^2  \nabla r+2 \rho \nabla \rho\right)\right)\\
&=& \nonumber \mathrm{O}(e^{(1-\delta)r}),
\end{eqnarray}
here we used $\rho=\mathrm{O}(e^{-(1-\delta)r})$.
\begin{multline}
\label{laplacian u}
\Delta u = \chi ''(e^{(2(1-\delta)r(x)}\rho_x^2(y)) \cdot \left ( e^{2\left(1-\delta \right)r}\left(\left(2-2\delta\right)\rho^2  \nabla r+2 \rho \nabla \rho \right)\right)^2
+\chi'(e^{(2(1-\delta)r(x)}) \\ \cdot \left(e^{2\left(1-\delta\right)r}\left(\left(2-2\delta\right)^2 \rho^2 |\nabla r|^2 \right. \right. 
+4\left. \left. \left(2-2\delta\right)\rho \nabla r \cdot \nabla \rho+\left(2-2\delta\right) \rho^2 \Delta r + 2|\nabla \rho|^2 +2 \rho \Delta\rho\right)\right).
\end{multline}

We need the following Hessian comparison theorem from \cite{MR1333601} to estimate $\Delta r$ and $\Delta \rho$.

\begin{theorem}
\label{Hessian}
Let $M_1$ and $M_2$ be two $n$-dimensional complete Riemannian manifolds. Assume that $\gamma_i:[0,a] \to M_i (i=1,2)$ are two geodesics parametrized by arc length, and $\gamma_i$ does not intersect the cut locus of $\gamma_i(0)$ for $i=1,2$. Let $r_i$ be the distance function from $\gamma_i(0)$ on $M_i$ and let $K_i$ be the sectional curvature of $M_i$. Assume that at $\gamma_1(t)$ and $\gamma_2(t)$, $0\leq t \leq a$, we have
$$K_1(X_1, \frac{\partial}{\partial \gamma_1}) \geq K_2(X_2, \frac{\partial}{\partial \gamma_2}) ,$$
where $X_i$ is any unit vector in $T_{\gamma_i(t)}M_i$ perpendicular to $\dfrac{\partial}{\partial \gamma_i}$. Denote by $H(r_i)$ the Hessian of $r_i$, then
$$H(r_1)(X_1,X_1)\leq H(r_2)(X_2,X_2),$$
where $X_i \in T_{\gamma_i(a)}M_i$ with $\langle X_i,\dfrac{\partial}{\partial \gamma_i} \rangle(\gamma_i(a))=0$ and $|X_i|=1$.
\end{theorem}

Since $\Delta r$ is the trace of $H(r)$, we have the following corollary.
\begin{corollary}
\label{Laplacian}
Let $M$ be an $n$-dimensional complete Riemannian manifold. If the sectional curvature satisfies $-k^2 \leq K_M(x) \leq -1$ in the geodesic ball $\overline{B_p(R)}$, then
$$(n-1)\,\mathrm{ coth }\, r \leq \Delta r \leq (n-1)k\, \mathrm{ coth }\, kr$$
for $r\leq R$. In addition, 
$$n-1 \leq \Delta r \leq (n-1)(k+\frac{1}{r})$$
for $r \leq R$.
\end{corollary} 

Since $-C e^{(2-2\delta)r(x)} \leq K_M(x)\leq -1$ in $B_p(r(x))$. By Corollary \ref{Laplacian} we have $n-1 \leq \Delta r(x) \leq (n-1)(1+C^{1/2} e^{(1-\delta)r(x)})$ for $r(x) \geq 1$. So 
\begin{equation}
\label{delta r}
\Delta r=\mathrm{O}(e^{(1-\delta)r}).
\end{equation} 

Since $-C e^{(2-2\delta)(r(x)+1)} \leq K_M(x)\leq -1$ for $x \in B_x(d(x)) \subset B_p(r(x)+1)$. By Corollary \ref{Laplacian}  again we have 
\begin{equation}
\label{delta rho}
\rho \Delta \rho \leq (n-1)(1+C e^{(1-\delta)(r+1)}\rho)=\mathrm{O}(1)
\end{equation}
when $d(x)/2 \leq \rho \leq d(x)$. Apply \eqref{delta r} and \eqref{delta rho} in \eqref{laplacian u} and use the fact that $\mathrm{supp}\, \Delta u \subset \overline{B_x(d(x))} \setminus B_x(d(x)/2)$ and $|\nabla r|=|\nabla \rho|=1$ we see that
\begin{equation}
\label{Delta u}
\Delta u=\mathrm{O}(e^{2(1-\delta)r}).
\end{equation}

To estimate $\nabla v$ we have
\begin{eqnarray*}
|\nabla v| &=& |\nabla \int_M u dy| \\
                &\leq& \int_M |\nabla u| dy \\
                &=& \int_{B_x(d(x))} |\nabla u| dy,
\end{eqnarray*}
thus 
\begin{equation}
\label{nabla v}
|\nabla v|=\mathrm{Vol}(B_x(d(x)) \cdot \mathrm{O}(e^{(1-\delta)r}).
\end{equation}

We also have
\begin{eqnarray*}
\label{Delta v}
|\Delta v| &=& |\Delta \int_M u dy| \\
                &\leq& \int_M |\Delta u| dy \\
                &=& \int_{B_x(d(x))} |\Delta u| dy,
\end{eqnarray*}
thus 
\begin{equation}
|\Delta v|=\mathrm{Vol}(B_x(d(x)) \cdot \mathrm{O}(e^{2(1-\delta)r}).
\end{equation}

Combining \eqref{nabla u}, \eqref{Delta u}, \eqref{nabla v} and \eqref{Delta v}, we have the following lemma.

\begin{lemma}
\label{laplacian uv}
\begin{equation}
\Delta (\frac{u}{v})=\left(\frac{1}{\mathrm{Vol}\left(B_x\left(d/2\right)\right)}+\frac{\mathrm{Vol}(B_x(d))}{\left(\mathrm{Vol}\left(B_x\left(d/2\right)\right)\right)^2}+\frac{(\mathrm{Vol}(B_x(d)))^2}{(\mathrm{Vol}(B_x(d/2)))^3}\right)  \cdot \mathrm{O}(e^{2(1-\delta)r}).
\end{equation}
\end{lemma}

To estimate $\Delta \overline{\varphi}(x)$, we need the following corollary of Bishop volume comparison theorem \cite{MR2243772}.

\begin{corollary}
\label{volume}
Let $M$ be a complete Riemannian manifold, and $c>0$ a constant. If $K_M(x) \geq -c^2k^2$  on $B_p(1)$ for some $k\geq1$. Then $\dfrac{\mathrm{Vol}(B_p(\frac{1}{k}))}{\mathrm{Vol}(B_p(\frac{1}{2k}))}\leq C_n$, where $C_n$ is a constant that depends only on the dimension of $M$ and $c$.
\end{corollary}

\begin{proof}
By Bishop Volume Comparison theorem, $\dfrac{\mathrm{Vol}(B_p(R))}{V(-c^2k^2,R)}$ is non-increasing in $R$ for $R \leq 1$, where $V(-c^2k^2,R)$ is the volume of the geodesic balls of radius $R$ in the space form of constant curvature $-c^2k^2$. Thus
$$\frac{\mathrm{Vol}(B_p(\frac{1}{k}))}{V(-c^2k^2,\frac{1}{k})} \leq \frac{\mathrm{Vol}(B_p(\frac{1}{2k}))}{V(-c^2k^2,\frac{1}{2k})}, $$
which can be written as
$$\frac{\mathrm{Vol}(B_p(\frac{1}{k}))}{\mathrm{Vol}(B_p(\frac{1}{2k}))} \leq \frac{V(-c^2k^2,\frac{1}{k})}{V(-c^2k^2,\frac{1}{2k})}.$$

In the hyperbolic space of constant curvature $-K^2$, the volume of a ball of radius $r$ is given by
\begin{equation}
\label{2}
V(-K^2, r)=\Omega_n (\frac{1}{K})^{n-1} \int_0^r \sinh^{n-1}(Kr) dr,
\end{equation}
where $\Omega_n$ is the surface area of the unit sphere in $\mathbb{R}^n$.

Computing using \eqref{2}
\begin{equation*}
V(-c^2 k^2,\frac{1}{k})=\Omega_n (\frac{1}{ck})^{n-1} \int_0^{\frac{1}{k}} \sinh^{n-1}(ckr) dr=\Omega_n  (\frac{1}{ck})^n \int_0^c \sinh^{n-1} r dr,
\end{equation*}
and
$$V(-c^2 k^2,\frac{1}{2k})=\Omega_n (\frac{1}{ck})^{n-1} \int_0^{\frac{1}{2k}} \sinh^{n-1}(ckr) dr=\Omega_n  (\frac{1}{ck})^n \int_0^{c/2} \sinh^{n-1} r dr.$$

Now we can take $C_n=\dfrac{V(-c^2 k^2,\frac{1}{k})}{V(-c^2 k^2,\frac{1}{2k})}=\dfrac{\int_0^c \sinh^{n-1} r dr }{\int_0^{c/2} \sinh^{n-1} r dr}$, which is a constant that depends only on $n$ and $c$.
\end{proof}

We are now ready to estimate $\Delta \overline{\varphi}(x)$.
\begin{eqnarray*}
|\Delta \overline{\varphi}(x)|&=&| \int_M \Delta (\frac{u}{v})(\varphi(y)-\varphi(x))  dy|\\
&\leq& \int_{B_x(d(x))} |\Delta(\frac{u}{v})| dy \cdot    \mathrm{osc}_{B_x(d(x))} \varphi \\
&\leq& \sup_{B_x(d(x))}\{|\Delta(\frac{u}{v})|\}\cdot \mathrm{Vol}(B_x(d(x))) \cdot    \mathrm{osc}_{B_x(d(x))} \varphi \\
&=& \left(\frac{\mathrm{Vol}(B_x(d))}{\mathrm{Vol}\left(B_x\left(d/2\right)\right)}+\frac{(\mathrm{Vol}(B_x(d)))^2}{\left(\mathrm{Vol}\left(B_x\left(d/2\right)\right)\right)^2}+\frac{(\mathrm{Vol}(B_x(d)))^3}{(\mathrm{Vol}(B_x(d/2)))^3}\right) \cdot \mathrm{O}(e^{-\delta r}).
\end{eqnarray*}

Observe that $B_x(d(x)) \subset B_p(r(x)+1)$ and on  $B_p(r(x))$, $K_M(x) \geq -Ce^{(2-2\delta)r(x)}=-C(\frac{1}{d(x)})^2$. By Corollary \ref{volume}, we have
\begin{equation}
\dfrac{\mathrm{Vol}(B_x(d(x)))}{\mathrm{Vol}(B_x(d(x)/2))} =\mathrm{O}(1).
\end{equation}

It follows that
\begin{equation}
\label{Delta varphi}\Delta \overline{\varphi}=\mathrm{O}(e^{-\delta r}).
\end{equation}

Define $g(x)=e^{-\delta_0 r(x)}$, where $\delta_0$ is a positive constant to be chosen later. We have
$$\Delta g=g \cdot (-\delta_0  \Delta r+\delta_0^2|\nabla r|^2).$$

Since $K_M \leq -1$ we have $\Delta r \geq n-1$, choose $\delta_0<\delta$ small enough so that 
\begin{equation}
\label{Delta negative}-\delta_0  \Delta r+\delta_0^2|\nabla r|^2<-(n-1)\delta_0+\delta_0^2
\end{equation}
is less that a negative constant.

Now, since $\Delta \overline\varphi=\mathrm{O}(e^{-\delta r})=\mathrm{o}(e^{-\delta_0 r})=\mathrm{o}(g)$, there exists a constant $\alpha>0$ such that
$$\Delta(\alpha g) \leq -|\Delta \overline{\varphi}|,$$which implies that $\overline\varphi-\alpha g$ is subharmonic and $\overline\varphi+\alpha g$ is superharmonic. It follows from the classical Perron's method that there exists a harmonic function $f$ such that
$$\overline{\varphi}-\alpha g \leq f \leq \overline{\varphi} +\alpha g.$$

Since $\overline{\varphi}$ and $\varphi$ have the same boundary value and $g=0$ on $S(\infty)$, $f=\varphi$ on the boundary. This completes the proof of Theorem \ref{main}.\\

\section{Martin Boundary}

Throughout this section we still assume $M$ is a complete, simply connected n-dimensional Riemannian manifold whose sectional curvature satisfies
$$-Ce^{(2-2\delta)r(x)} \leq K_M(x) \leq -1,$$
but everything carries over to manifolds which admit positive superharmonic functions vanishing at infinity.

From Theorem \ref{main} we know there exists a nontrivial bounded harmonic function $f$ on $M$. This implies (cf. \cite{MR1333601}) that $M$ possesses a positive symmetric Green's function $G(p,x)$. Moreover, if we denote by $G_i(p,x)$ the Green's function  on $\Omega_i$ with Dirichlet boundary condition, where $\{\Omega_i,i=1,2,\cdots\}$ is a compact exhaustion of $M$, then $G_i$ converges uniformly to $G$ on compact subsets of $M \setminus\{p\}$.

We have shown on page \pageref{Delta negative} that if $\alpha>0$ is sufficiently small, then 
$$\Delta(e^{-\alpha r}) \leq 0$$ on $M$. Let $$C_1=\sup_{\partial B_p(1)} G(p,x)e^{\alpha r(x)}>0,$$ we have
$$G_i(p,x) \leq G(p,x) \leq C_1 e^{-\alpha r(x)} \textrm{ on } \partial B_p(1),$$
$$0=G_i(p,x) <  C_1 e^{-\alpha r(x)} \textrm{ on } \partial \Omega_i,$$ and
$$0=\Delta G_i \geq \Delta (C_1 e^{-\alpha r}) \textrm{ on } \Omega_i \setminus B_p(1).$$
It follows from the maximum principle that
$$G_i \leq C_1 e^{-\alpha r} \textrm{ on } \Omega_i \setminus B_p(1).$$
Passing to the limit
$$G(p,y) \leq C_1 e^{-\alpha r(x)} \textrm{ on } M \setminus B_p(1),$$
which implies that $G$ extends continuously to $\overline{M}$ with value $0$ on $S(\infty)$.

For $x,y \in M$, let
\begin{displaymath}
h_y(x)=\dfrac{G(x,y)}{G(p,y)}
\end{displaymath}
be the normalized Green's function with $h_y(p)=1$. A sequence $Y=\{y_i\}$ is called fundamental if $h_{y_i}$ converges to a positive harmonic function $h_Y$ on $M$. Two fundamental sequences $Y$ and $\overline{Y}$ are said to be equivalent if the corresponding limiting positive harmonic functions $h_Y$ and $h_{\overline{Y}}$ are the same.

\begin{definition}
The Martin boundary $\mathcal{M}$ of $M$ is the set of equivalence classes of non-convergent fundamental sequences.
\end{definition}

Let $\widetilde{M}=M \cup \mathcal{M}$. For each $y \in M$, all sequences  converging to $y$ form an equivalence class $[Y]$. On the other hand, two fundamental sequences that have different limit points in $M$ are not equivalent. Thus $\widetilde{M}$ can be identified with the set of  equivalence classes of fundamental sequences. Define a metric $\rho$ on $\widetilde{M}$
\begin{equation}
\rho([Y],[Y'])=\sup_{B_p(1)} |h_Y(x)-h_{Y'}(x)|
\end{equation}
for $[Y],[Y'] \in \widetilde{M}$. The topology induced by $\rho$ makes $\widetilde{M}$ a compactification of $M$.

It is known from \cite{MR1333601} that if for all $\theta_1, \theta_2$ with $0<\theta_2<\theta_1<\pi/4$, there exists a positive constant $\alpha$ depending only on $n$, $C$, $\delta$, $\theta_1$ and $\theta_2$, such that for any positive harmonic function $u \in C^0(\overline{C_p(\theta_1)})$ which vanishes on $\overline{C_p(\theta_1)} \cap S(\infty)$, the Harnack inequality
\begin{equation}
\label{harnack 1} u(x) \leq C_1 u(p') e^{-\alpha r(x)}
\end{equation}
holds on $T(\theta_2,1)$ , then there is a natural surjection $\Phi: \mathcal{M} \to S(\infty)$. In fact, let $\{y_k\}$ be a sequence of points converging to $\xi \in S(\infty)$. ``1 It follows that $P_\xi \neq P_{\tilde{\xi}}$ if $\xi \neq \tilde{\xi}$. Thus a fundamental sequence has a unique limit point. The map is then well defined and surjective.

Moreover, if for any positive harmonic functions $u,v \in C^0(\overline{C_p(\theta_1)})$ which vanish on $\overline{C_p(\theta_1)} \cap S(\infty)$, we have, for all $x \in T(\theta_2,1)$,
\begin{equation}
\label{harnack 2}\tilde{C}^{-1}\frac{u(p')}{v(p')} \leq \frac{u(x)}{v(x)} \leq \tilde{C}\frac{u(p')}{v(p')}, 
\end{equation}
then $\Phi$ defined above is one-to-one and therefore a homeomorphism. For further details, see Chapter II in \cite{MR1333601}.\\

\section{Boundary Harnack Inequalities}

In this section we prove \eqref{harnack 1} and \eqref{harnack 2} to establish homeomorphism between $\mathcal{M}$ and $S(\infty)$. We assume $M$ is a complete, simply connected $n$-dimensional Riemannian manifold whose sectional curvature satisfies $$-Ce^{(2/3-2\delta)r} \leq K_M \leq -1,$$unless otherwise stated.

Given $\omega \in S_p$. Let $p'=\mathrm{exp}_p \omega$. Recall that $C_p(\theta)=C_p(\omega,\theta)$ is the cone about $\omega$ of angle $\theta$ at $p$, and $T_p(\theta,R)=T_p(\omega,\theta,R)=C_p(\omega,\theta)\setminus B_p(R)$ is the truncated cone.

Let $0<\theta_2<\theta_1<\pi/4$ and $\theta_3=(\theta_1+\theta_2)/2$.

We want to prove the following two boundary Harnack inequalities.

\begin{theorem}
\label{Harnack 1}
Let $u$ be a positive harmonic function on $C_p(\theta_1)$ which is continuous on  $\overline{C_p(\theta_1)}$ and  vanishes on $\overline{C_p(\theta_1)} \cap S(\infty)$. Then for all $x \in T(\theta_2,1)$,
\begin{displaymath}
u(x) \leq \tilde{C}e^{-\alpha r(x)}u(p'),
\end{displaymath}
where  $\tilde{C}$ and $\alpha$ depend only on $n$, $C$, $\delta$, $\theta_1$ and $\theta_2$.
\end{theorem}

\begin{theorem}
\label{Harnack 2}
Let $u$, $v$ be two positive harmonic functions on  $C_p(\theta_1)$ which are continuous on  $\overline{C_p(\theta_1)}$ and which vanish on $\overline{C_p(\theta_1)} \cap S(\infty)$. Then for all $x \in T(\theta_2,1)$,
\begin{displaymath}
 \tilde{C}^{-1}\frac{u(p')}{v(p')} \leq \frac{u(x)}{v(x)} \leq \tilde{C}\frac{u(p')}{v(p')},
\end{displaymath}
where $\tilde{C}$ depends only on $n$, $C$, $\delta$, $\theta_1$ and $\theta_2$.
\end{theorem}

First we need to construct a cut-off function with small second derivatives.
\begin{lemma}
\label{cut-off}
Given two constants $\alpha$ and $\beta$, there exists $\varphi \in C^{\infty}(M)$ and a constant $R_0>0$ such that
\begin{equation*}
\left \{
\begin{array}{rcl}
\varphi &=& \alpha  \text{ on } T(\theta_2,R_0),\\
\varphi &=& \beta \text{ on } \partial C_p(\theta_1)\setminus B_p(R_0),\\
|\nabla \varphi| &=& \mathrm{O}(e^{-r}) \text{ on } T(\theta_2,R_0),\\
|\Delta \varphi| &=& \mathrm{O}(e^{-(2/3+\delta)r}) \text{ on } T(\theta_2,R_0).
\end{array} \right.
\end{equation*}
\end{lemma}

\begin{proof}
Let $\psi \in C^\infty([0,\pi])$ be a function satisfying $0 \leq \psi \leq 1$, $\psi(t)=0$ for $t \in [0,\theta_2+\epsilon]$ and $\psi(t)=1$ for $t \in [\theta_1-\epsilon,\theta_1+\epsilon]$, where $\epsilon<(\theta_1-\theta_2)/2$ is a small positive constant. Set  $$\tilde{\psi}(x)=\psi(\angle(px,v)).$$ $\tilde{\psi}$ is smooth and bounded on $M\setminus \{p\}$. We take the average $\varphi$ of $\tilde{\psi}$ in the ball $B_x(e^{(-1/3+\delta)r(x)})$ by defining
$$\varphi(x)=\dfrac{\int_M \chi(e^{(2/3-2\delta)r(x)}\rho_x^2(y)) \tilde{\psi}(y) dy}{\int_M \chi(e^{(2/3-2\delta)r(x)} \rho_x^2(y)) dy},$$
where $\chi \in C_0^\infty(\mathbb{R})$ is a cut-off function satisfying $0 \leq \chi \leq 1$, $\chi(t)=0$ for $|t| \geq 1$ and $\chi(t)=1$ for $|t| \leq \dfrac{1}{4}$.  The proof that $\varphi$ is our desired cut-off function is very similar to that of Theorem \ref{main} on page \pageref{Delta varphi}.
\end{proof}

We will need the following gradient estimate for positive harmonic functions due to Yau.

\begin{theorem}(\cite{MR0431040})
\label{gradient estimate}
Let $N$ be a complete Riemannian manifold of dimension $n$. Suppose that the Ricci curvature on $B_p(R)$ is bounded from below by $-(n-1)K$ for some constant $K \geq 0$. If $u$ is a positive harmonic function on $B_p(R)$, then for any $0<\epsilon<1$, we have, for all $x \in B_{\epsilon R}$,
\begin{displaymath}
\dfrac{|\nabla u|}{u} \leq \tilde{C}(\dfrac{1}{R}+\sqrt{K}),
\end{displaymath}
where $\tilde{C}$ is a constant depending only on $n$ and $\epsilon$.
\end{theorem}

Applying Theorem \ref{gradient estimate} on our manifold $M$ we get the following corollary.

\begin{corollary}
\label{local gradient estimate}
Let $M$ be a complete, simply connected $n$-dimensional Riemannian manifold with $-Ce^{(2/3-2\delta)r} \leq K_M \leq -1$. If $u$ is a positive harmonic function on $M$, then 
\begin{displaymath}
|\dfrac{\nabla u}{u}|(x) \leq \tilde{C}e^{(1/3-\delta/2)r(x)},
\end{displaymath}
where $\tilde{C}$ depends only on $n$, $C$ and $\delta$.
\end{corollary}

\begin{proof}
For every $x \in M$, let 
$$R=\dfrac{2/3-\delta}{2/3-2\delta}\cdot r(x)=C_1 r(x),$$
with $C_1>1$.
We have
\begin{displaymath}
K_M \geq -C e^{(2/3-2\delta)R}=-C e^{(2/3-\delta)r} 
\end{displaymath}
on $B_p(R)$.

Apply Theorem \ref{gradient estimate} with $\epsilon=1/C_1$ to obtain
\begin{displaymath}
|\dfrac{\nabla u}{u}| \leq \tilde{C}e^{(1/3-\delta/2)}
\end{displaymath}
on $B_p(R/C_1)=B_p(r(x))$.
\end{proof}

\begin{lemma}
\label{weak harnack}
Let $u$ be a positive harmonic function on $C_p(\theta_3)$ which is continuous on  $\overline{C_p(\theta_3)}$ and which vanishes on $\overline{C_p(\theta_3)} \cap S(\infty)$. Then for all $x \in T(\theta_2,1)$,
\begin{displaymath}
u(x) \leq \tilde{C}e^{-\alpha r(x)} \sup_{\partial C_p(\theta_3)} u,
\end{displaymath}
where $\alpha$ is a constants depending only on $n$, $C$, $\delta$, $\theta_1$ and $\theta_2$.
\end{lemma}

\begin{proof}
By Lemma \ref{cut-off}, there exists $\varphi \in C^{\infty}(M)$ and a constant $R_1>0$ such that
\begin{equation*}
\left \{
\begin{array}{rcl}
\varphi &=& 0  \text{ on } T(\theta_2,R_1),\\
\varphi &=& 1 \text{ on } \partial C_p(\theta_3)\setminus B_p(R_1),\\
|\nabla \varphi| &=& \mathrm{O}(e^{-r}) \text{ on } T(\theta_3,R_1),\\
|\Delta \varphi| &=& \mathrm{O}(e^{-(2/3+\delta)r}) \text{ on } T(\theta_3,R_1).
\end{array} \right.
\end{equation*}

Consider $e^{-\alpha r}$, where $\alpha<\dfrac{2}{3}+\delta$ is sufficiently small. We have
\begin{eqnarray*}
\Delta e^{-\alpha r}&=&e^{-\alpha r}(\alpha^2-\alpha \Delta r)\\
&\leq& e^{-\alpha r}(\alpha^2-(n-1)\alpha)<0.
\end{eqnarray*}

Since $\Delta \varphi=\mathrm{O}(e^{-(2/3+\delta)r})$, we have $|\Delta \varphi|<-C_1 \Delta   e^{-\alpha r}$ on $T(\theta_3,R_0)$ for $C_1$ and $R_0$ sufficiently large.

Let $f=\varphi +\tilde{C}e^{-\alpha r}$. We have $\Delta f \leq 0$ on $T(\theta_3,R_0)$ if $\tilde{C}>C_1$. Also $f \geq 1$ on $\partial T(\theta_3,R_0)$, provided $\tilde{C}$ is sufficiently large.

Now consider $\bar{u}=u / \sup_{\partial C_p(\theta_3)} u$. $\bar{u}$ is harmonic and $\bar{u} \leq 1$ on $C_p(\theta_3)$. We have $\Delta(\bar{u}-f)\geq 0$ on $T(\theta_3, R_0)$ and $\bar{u}-f \leq 0$ on $\partial T(\theta_3, R_0)$. By the maximum principle, $\bar{u} \leq f$ on $T(\theta_3, R_0)$. In particular,
$$u(x) \leq f(x) \sup_{\partial C_p(\theta_3)} u=\tilde{C}e^{-\alpha r} \sup_{\partial C_p(\theta_3)} u$$
for all $x \in T(\theta_2,R_0)$.

The truncated cones $T(\theta_2,R_0)$ and $T(\theta_2,1)$ differ by a precompact set, by the Harnack inequality the estimate holds on $T(\theta_2,1)$ with a larger $\tilde{C}$.
\end{proof}

We are now ready to prove Theorem \ref{Harnack 1}. By Lemma \ref{weak harnack}, it is sufficient to show that harmonic functions satisfying the given conditions and $u(p')=1$ are uniformly bounded on $\partial C_p(\theta_3)$. In the following we will use $C_1, C_2, \dots$, $\alpha_1, \alpha_2, \dots$ and $R_1, R_2, \dots$ to denote  positive constants depending only on $n$, $C$, $\delta$, $\theta_1$ and $\theta_2$.

By Lemma \ref{cut-off}, there exists $\varphi \in C^{\infty}(M)$ with $\dfrac{2}{3} \leq \varphi \leq 1$ and a constant $R_0>0$ such that
\begin{equation*}
\left \{
\begin{array}{rcl}
\varphi &=& \dfrac{2}{3}  \text{ on } T(\theta_3,R_0),\\
\varphi &=& 1 \text{ on } \partial C_p(\theta_1) \setminus B_p(R_0),\\
|\nabla \varphi| &=& \mathrm{O}(e^{-r}) \text{ on } T(\theta_1,R_0),\\
|\Delta \varphi| &=& \mathrm{O}(e^{(-2/3-\delta)r}) \text{ on } T(\theta_1,R_0).
\end{array} \right.
\end{equation*}

Consider the function $u^\varphi$. Direct computation gives
\begin{equation}
\nabla u^\varphi=u^\varphi(\log u \nabla \varphi+\varphi \nabla \log u),
\end{equation}
\begin{equation}
\Delta u^\varphi=u^\varphi(|\log u \nabla \varphi+\varphi \nabla \log u|^2+\log u \Delta \varphi+2 \nabla \varphi \cdot \nabla \log u+\varphi \Delta \log u). \label{Delta}
\end{equation}

Using Corollary \ref{local gradient estimate}, we have
\begin{equation}
|\nabla \log u|=\mathrm{O}(e^{(1/3-\delta/2)r}) \label{gradient log u}
\end{equation}
and
\begin{eqnarray}
\label{log u}|\log u(x)|&=&|\log u(x) -\log u(p')|\\
\nonumber &\leq& \int_\gamma |\nabla \log u|\\
\nonumber&\leq& \sup_{B_p(r(x))} |\nabla \log u| \cdot  d(x,p')\\
\nonumber&\leq& \tilde{C}e^{(1/3-\delta/2)r(x)}(r+1)\\
\nonumber&=& \mathrm{o}(e^{r/3}),
\end{eqnarray}
where $\gamma$ is the geodesic segment connecting $x$ and $p'$.

Observe that
$$\Delta \log u=\frac{\Delta u}{u}-\frac{|\nabla u|^2}{u^2}=-|\nabla \log u|^2.$$
Therefore
\begin{equation}
\label {Delta u varphi} \Delta u^\varphi \leq u^\varphi(C_1 e^{(-1/3-\delta/2)r}+(\varphi^2-\varphi)|\nabla \log u|^2).
\end{equation}

Let $\psi \in C^\infty(\mathbb{R})$ be a function such that
\begin{equation*}
\left \{
\begin{array}{rl}
&1 \leq \psi \leq C_2,\\
&\psi'(t)=-\dfrac{1}{|t|\log^2 |t|} \text{ for } |t| \geq R_1,\\
&-\psi'(t) \geq 12|\psi''(t)| \geq 0 \text{ for } \text{ for all } t.
\end{array} \right.
\end{equation*}
Such a function could be constructed by elementary calculas.

Let $\xi=\log u$. Set
$$F(x)=\psi(\xi(x)-e^{-\beta r(x)}) \cdot u^\varphi,$$
where $\beta$ is a positive number to be determined later. We have for $\psi=\psi(\xi(x)-e^{-\beta r(x)})$,
$$\nabla \psi=\psi' \cdot (\nabla \xi+\beta e^{-\beta r} \nabla r)  ,$$
\begin{eqnarray*}
\Delta \psi&=&\psi'' \cdot |\nabla \xi+\beta e^{-\beta r} \nabla r|^2+\psi'\cdot (-|\nabla \xi|^2+ e^{-\beta r} (\beta \Delta r-\beta^2)).
\end{eqnarray*}
Therefore
\begin{eqnarray*}
\Delta F &=&\psi \Delta u^\varphi +\Delta \psi u^\varphi  +2\nabla \psi \cdot \nabla u^\varphi\\
&=& \psi \Delta u^\varphi +\psi'' u^\varphi\cdot |\nabla \xi+\beta e^{-\beta r} \nabla r|^2\\
&&+\psi' u^\varphi \cdot ((2\varphi-1)|\nabla \xi|^2+2\xi \nabla \varphi \cdot \nabla \xi +2\beta  e^{-\beta r} \xi \nabla \varphi \cdot \nabla r\\
&&+2\beta \varphi  e^{-\beta r} \nabla \xi \cdot \nabla r+ e^{-\beta r} (\beta \Delta r-\beta^2)).
\end{eqnarray*}

Using \eqref{gradient log u} \eqref{log u} and \eqref{Delta u varphi}, we obtain the following estimate
\begin{eqnarray}
\nonumber \Delta F &\leq& u^\varphi[C_1 C_2 \psi e^{(-1/3-\delta/2) r} +2\beta^2 e^{-2\beta r}|\psi''|+\psi' \cdot( -C_3e^{(-1/3-\delta/2) r}\\
\nonumber&&+e^{-\beta r} (\beta \Delta r-\beta^2-\beta \varphi))]+u^\varphi|\nabla \xi|^2[2|\psi''|+(2\varphi-1-\beta \varphi e^{-\beta r})\psi' ]\\
\nonumber&\leq& u^\varphi (C_4 e^{(-1/3-\delta/2) r}+\psi' \cdot (\beta \Delta r-\beta^2-\beta) e^{-\beta r})+ u^\varphi \psi' |\nabla \xi|^2(\dfrac{1}{6}-\beta e^{-\beta r}).
\end{eqnarray}
Here we have used the inequalities $2/3\leq \varphi \leq 1$ and $\nabla \xi \cdot \nabla r \leq |\nabla \xi|^2+1$.

Since $\Delta r \geq n-1$, we can take $\beta<\delta/4$ sufficiently small so that
\begin{equation}
 (\beta \Delta r-\beta^2-\beta) e^{-\beta r} \geq C_5 e^{-\beta r}.
\end{equation}

At points $x$ such that $\xi=\xi(x) \geq R_1+1$, from \eqref{log u} we have
\begin{equation}
\xi=\mathrm{o}(e^{r/3}) \textrm{  and  } \log \xi=\mathrm{O}(r),
\end{equation}
Together with the definition of $\psi$, this yields
\begin{eqnarray*}
\psi' \cdot (\beta \Delta r-\beta^2-2\beta) e^{-\beta r} &\leq& -\dfrac{1}{|\xi|\log^2|\xi|} C_5 e^{-\beta r}\\
&\leq&-C_6 \dfrac{1}{r^2} e^{-(1/3+\beta)r}\\
&\leq& -C_7 e^{-(1/3+\delta/4)r}
\end{eqnarray*}
for $r$ sufficiently large. Here we have used that $e^{-(1/3+\beta)r}/r^2=O(e^{-(1/3+\delta/4)r})$ since $\beta<\delta/4$.

It follows that for $\xi\geq R_1+1$, we have 
\begin{eqnarray}
\label{Delta F} \Delta F &\leq& u^\varphi(C_4 e^{(-1/3-\delta/2) r}-C_7 e^{-(1/3+\delta/4)r})\\
\nonumber&&+u^\varphi \psi' |\nabla \xi|^2(\dfrac{1}{6}-2\beta e^{-\beta r}).
\end{eqnarray}
If $R_2$ is sufficiently large then $C_4 e^{(-1/3-\delta/2) r}-C_7 e^{-(1/3+\delta/4)r}<0$ and $\dfrac{1}{6}-2\beta e^{-\beta r}>0$ for $r \geq R_2$, so that $\Delta F \leq 0$ on $T(\theta_1,R_2)$.

The remaining case is when $\xi \leq R_1+1$. We then have $u=e^{\xi}\leq C_8=e^{R_1+1}$ is bounded. In addition, $u^\varphi \leq C_8 u^{1/2}$ and $u^\varphi |\xi| \leq C_9 u^{1/2}$. Using the fact that $|\psi|$, $|\psi'|$ and $|\psi''|$ are all bounded, we conclude that
\begin{displaymath}
\Delta F \leq C_{10} u^{1/2}|\nabla \xi|^2+C_{11}e^{-\beta r}.
\end{displaymath}

Define
\begin{equation}
G=F+C_{12} u^{1/2}+e^{-\alpha_1 r}.
\end{equation}
This is a positive function with $\alpha_1<\beta$ and $C_{12}$ to be determined. It is clear that $G_F=C_{12} u^{1/2}+e^{-\alpha_1 r}$ is superharmonic. Therefore
\begin{eqnarray}
\label{Delta G} \Delta G&=&\Delta F +C_{12}u^{1/2}(-\dfrac{1}{4}|\nabla \xi|^2)+e^{-\alpha r}(\alpha^2-\alpha \Delta r)\\
\nonumber&\leq& C_{10} u^{1/2}|\nabla \xi|^2+C_{11}e^{-\beta r})-\dfrac{C_{12}}{4}u^{1/2}|\nabla \xi|^2-C_{13} e^{-\alpha r}\\
\nonumber&=& (C_{10}-\dfrac{C_{12}}{4})u^{1/2}|\nabla \xi|^2+C_{11}e^{-\beta r}-C_{13} e^{-\alpha r}\\
\nonumber&\leq& 0
\end{eqnarray}
on $T(\theta_1,R_3)$ if $C_{12},C_{13}$ and $R_3$ are sufficiently large.

Combine \eqref{Delta F}, \eqref{Delta G} and the superharmonicity of $G-F$, we have for all $\xi$, $$\Delta G \leq 0$$ on $T(\theta_1,R_0)$ with $R_0=\max(R_2,R_3)$.

Since $u$ is harmonic, $$ \Delta (C_{14} G-u) \leq 0$$  on $T(\theta_1,R_0)$, where $C_{14}>1$ is a constant to be determined. Observe that $F=\psi u \geq u$ on $\partial C_p(\theta_1)$. Therefore
\begin{equation*}
C_{14} G-u \geq C_{14} F-u \geq 0
\end{equation*} 
on $\partial C(\theta_1)\setminus B_p(R_3)$. If $C_{14}$ is sufficiently large, we also have 
\begin{equation*}
C_{14} G-u \geq 0 
\end{equation*}
on $\partial B_p(R_3) \cap C(\theta_1)$. By the maximum principle,
\begin{equation*}
C_{14} G \geq u 
\end{equation*} 
on $T(\theta_1,R_0)$. In particular, on $T(\theta_3,R_0)$ we have
\begin{eqnarray*}
u \leq C_{14}G &=&C_{14} \psi u^{2/3}+C_{12} C_{14} u^{1/2}+e^{-\alpha_1 r} \\
&\leq& C_1 C_{14}u^{2/3}+C_9 C_{10} u^{1/2}+e^{-\alpha_1 r},
\end{eqnarray*}
which implies that $u$ is bounded on $T(\theta_3,R_0)$. By the gradient estimate $u$ is also bounded on $\overline{C_p(\theta_3) \cap B_p(R_0)}$. Therefore, positive harmonic functions on $C_p(\theta_1)$ which vanish on $\overline{C_p(\theta_1)}\cap S(\infty)$ are uniformly bounded on $\overline{C_p(\theta_3)}$. Now applying Lemma \ref{weak harnack} we have for all $x \in T(\theta_2,R_0)$,
\begin{displaymath}
u(x) \leq \tilde{C}e^{-\alpha r(x)} u(p').
\end{displaymath}

By the gradient estimate for the harmonic function $u$ , the Harnack inequality above is true on $T(\theta_2,1)$ with a larger $\tilde{C}$. This completes the proof of Theorem \ref{Harnack 1}.

\begin{proof}[Proof of Theorem \ref{Harnack 2}]
Without loss of generality we may assume that $u(p')=v(p')=1$.

By Theorem \ref{Harnack 1}, we have
\begin{equation}
\label{u,v bound} u, v \leq C_1 e^{-\alpha_1 r}
\end{equation}
on $T(\theta_3,1)$.

Let $\xi=-\log u$. From the gradient estimate we have
\begin{equation}
\label{nabla xi}\nabla \xi=\mathrm{O}(e^{(1/3-\delta/2)r})
\end{equation}
and
\begin{equation}
\label{xi}\xi=\mathrm{o}(e^{r/3}).
\end{equation}
Thus we have 
\begin{equation}
\label{upper lower} C_2 r \leq \xi \leq C_3 e^{r/3}
\end{equation}
on $T(\theta_3,R_1)$. It follows from \eqref{u,v bound} and \eqref{xi} that
\begin{equation}
\xi^{-\epsilon} \geq e^{-\epsilon r/3} \geq C_4 v
\end{equation}
for $\epsilon>0$ sufficiently small.

We will construct a function $F \in C^\infty(C_p(\theta_3))$ satisfying

It will then follow from the maximum principle that $v \leq F$ on $T(\theta_3,R_0)$. In particular, $v \leq C_5 u$ on $T(\theta_2,R_0)$, which gives the first inequality in Theorem \ref{Harnack 2}. By exchanging $u$ and $v$ we get the second inequality immediately.

We now proceed to construct $F$ satisfying (i),(ii) and (iii). By Lemma \ref{cut-off}, there exists $\varphi \in C^{\infty}(M)$ with $0\leq \varphi \leq 1$ such that
\begin{equation}
\label{varphi}\left \{
\begin{array}{rcl}
\varphi &=& 0  \text{ on } T(\theta_2,R_1),\\
\varphi &=& 1 \text{ on } \partial C_p(\theta_3)\setminus B_p(R_1),\\
|\nabla \varphi| &=& \mathrm{O}(e^{-r}) \text{ on } T(\theta_3,R_1),\\
|\Delta \varphi| &=& \mathrm{O}(e^{-(2/3+\delta)r}) \text{ on } T(\theta_3,R_1).
\end{array} \right.
\end{equation}

Consider the function $f=u^{1-\varphi}\xi^{-\epsilon \varphi}$. We have 
\begin{eqnarray*}
f&=&u \text{ on } T(\theta_2,R_1),\\
f&=& \xi^{-\epsilon} \geq C_4 v \text{ on } \partial C_p(\theta_3)\setminus B_p(R_1),
\end{eqnarray*}
and
\begin{equation}
\nabla f=f\cdot(\xi \nabla \varphi-(1-\varphi)\nabla \xi -\epsilon \log \xi \nabla \varphi -\epsilon \varphi \nabla \log\xi),
\end{equation}
\begin{eqnarray*}
\Delta f&=&f \cdot (|\xi \nabla \varphi-(1-\varphi)\nabla \xi -\epsilon \log \xi \nabla \varphi -\epsilon \varphi \nabla \log\xi|^2\\
&&+\xi \Delta \varphi +2\nabla \varphi \cdot \nabla \xi-(1-\varphi)\Delta \xi-\epsilon \log \xi \Delta \varphi-2\epsilon \nabla \log \xi \cdot \nabla \varphi-\epsilon \varphi \Delta \log \xi).
\end{eqnarray*}

Observe that
\begin{equation}
\Delta \xi= -\dfrac{\Delta u}{u}+|\nabla \xi|^2=|\nabla \xi|^2,
\end{equation}
and
\begin{equation}
\Delta \log \xi=\dfrac{\Delta \xi}{\xi}-\dfrac{|\nabla \xi|^2}{\xi^2}=\dfrac{|\nabla \xi|^2}{\xi}-|\nabla \log \xi|^2.
\end{equation}
Therefore by  \eqref{nabla xi}, \eqref{upper lower} and \eqref{varphi} we have
\begin{equation}
\Delta f \leq f \cdot [(\varphi^2-\varphi) |\nabla \xi|^2+\epsilon \varphi(1-2\varphi)\dfrac{|\nabla \xi|^2}{\xi}+(\epsilon^2 \varphi^2+\epsilon \varphi)|\nabla \log \xi|^2+C_6 e^{-(1/3+\delta/2)r}].
\end{equation}

Let $\psi \in C^\infty(\mathbb{R})$ be a function such that
\begin{equation*}
\left \{
\begin{array}{rl}
&1 \leq \psi \leq C_7,\\
&\psi'(t)=\dfrac{1}{|t|\log^2 |t|} \text{ for } |t| \geq R_2,\\
&\psi'(t) \geq 12|\psi''(t)| \geq 0 \text{ for } \text{ for all } t.
\end{array} \right.
\end{equation*}

Set
\begin{equation}
F(x)=\psi(\xi+e^{-\beta r(x)}) \cdot f.
\end{equation}
We have for $\psi=\psi(\xi+e^{-\beta r(x)})$,
$$\nabla \psi=\psi' \cdot (\nabla \xi-\beta e^{-\beta r} \nabla r),$$
\begin{eqnarray*}
\Delta \psi&=&\psi'' \cdot |\nabla \xi-\beta e^{-\beta r} \nabla r|^2+\psi'\cdot (\Delta \xi+e^{-\beta r} (\beta^2-\beta \Delta r)).
\end{eqnarray*}

Then we have
\begin{eqnarray*}
\Delta F &=&\psi \Delta f +\Delta \psi f  +2\nabla \psi \cdot \nabla f\\
&=& \psi \Delta f + \psi'' f\cdot |\nabla \xi -\beta e^{-\beta r} \nabla r|^2 \\
&&+\psi'f \cdot (-|\nabla \xi|^2 +e^{-\beta r}(\beta^2-\beta \Delta r)) +2\psi' \cdot(\nabla \xi-\beta e^{-\beta r} \nabla r)\cdot \nabla f\\
&\leq& \psi f [(\varphi^2-\varphi) |\nabla \xi|^2+\epsilon \varphi(1-2\varphi)\dfrac{|\nabla \xi|^2}{\xi}+(\epsilon^2 \varphi^2+\epsilon \varphi)|\nabla \log \xi|^2\\
&&+C_6 e^{-(1/3+\delta/2)r}]+\psi' f[(2\varphi-3)|\nabla \xi|^2-2\epsilon \varphi \dfrac{|\nabla \xi|^2}{\xi}+2\beta (1-\varphi)e^{-\beta r}\nabla \xi \cdot \nabla r\\
&&+2\epsilon \beta \varphi e^{-\beta r} \nabla \log \xi \cdot \nabla r+C_8 e^{-\beta r} (\beta^2-\beta \Delta r)]\\
&\leq& \psi f [\varphi(\varphi-1) |\nabla \xi|^2+\varphi(\epsilon -2\epsilon \varphi)\dfrac{|\nabla \xi|^2}{\xi}+\varphi(\epsilon^2 \varphi+\epsilon)|\nabla \log \xi|^2\\
&&+C_6 e^{-(1/3+\delta/2)r}]+\psi' f[(2\varphi-3+2\beta(1-\varphi+\epsilon \varphi))|\nabla \xi|^2\\
&&+C_9 e^{-2\beta r}+C_8 e^{-\beta r} (\beta^2-\beta \Delta r)].
\end{eqnarray*}
Here we have used the inequalities $e^{-\beta r}\nabla \xi \cdot \nabla r \leq |\nabla\xi|^2+e^{-2\beta r}$ and $|\nabla \xi|^2/\xi \leq  |\nabla \xi|^2$.

As in the proof of Theorem \ref{Harnack 1}, we can choose $\beta<\delta/4$ to be sufficiently small and $R_0>\max\{R_1,R_2\}$ sufficiently large so that
\begin{eqnarray*}
\Delta F &\leq& f[C_6 C_7 e^{-(1/3+\delta/2)r}-C_{10}e^{-\beta r}/(\xi \log^2 \xi)]\\
&\leq& f[C_6 C_7 e^{-(1/3+\delta/2)r}-C_{11}e^{-(1/3+\delta/4)r}]\\
&\leq& 0
\end{eqnarray*}
on $T(\theta_3,R_0)$. This is possible because $\xi=\mathrm{o}(e^{r/3})$ and $\log \xi =\mathrm{o}(r)$.

We already know that
\begin{eqnarray*}
F= \psi f \geq C_4 v
\end{eqnarray*}
on $\partial C_p(\theta_3)\setminus B_p(R_0)$. Therefore $C_{12} F \geq v$ on $\partial T(\theta_3,R_0)$ if $C_{12}$ is sufficiently large. Since $v$ is harmonic, by the maximum principle, we have $C_{12} F \geq v$ on $\overline{T(\theta_3,R_0)}$. In particular,
\begin{displaymath}
v \leq C_{12} F \leq C_7 C_{12} u
\end{displaymath}
on $\overline{T(\theta_2,R_0)}$.

Since $u(p')=v(p')=1$, by the gradient estimate we have
$$C_{13} \leq u,v \leq C_{14}$$ on $B_p(R_0) \cap  T(\theta_2,1)$. Then $\dfrac{u}{v}\geq \tilde{C}^{-1}$ on $T(\theta_2,1)$ with $\tilde{C}=\max(\dfrac{C_{13}}{C_{14}},C_7 C_{12})$. 
\end{proof}
\emph{Remark 5.1} As remarked in Section 4, Theorem \ref{Martin} follows immediately from Theorem \ref{Harnack 1} and Theorem \ref{Harnack 2}.\\

\section{Proof of Theorem \ref{main2}}

In this section we will assume that on the complete, simply connected Riemannian manifold $M$, there exists $\alpha>2$ and $\beta<\alpha-2$ such that the following curvature conditions are satisfied:
$$K_M \leq -\dfrac{\alpha(\alpha-1)}{r^2}, \;\;\; Ric_M \geq -r^{2\beta},$$
where $r=d(p,\cdot)$ is the distance to the base point $p$.

\begin{lemma}
\label{angle}
Let $x,y$ be two points in $M$. Suppose that $d(p,x)=s$ and $y \in B_x(d)$ with $d<s$. We have $$\angle(px, py)<\dfrac{d}{(s-d)^\alpha}.$$
\end{lemma}

To prove Lemma \ref{angle} we apply the Topogonov comparison theorem on the rotationally symmetric model space $\mathbb{R}^2$ with pole $\tilde{p}$ and metric $\tilde{g}=dr^2+r^2\alpha d\phi^2$. It is easy to obtain that $(\mathbb{R}^2, \tilde{g})$ has sectional curvature $-\dfrac{\alpha(\alpha-1)}{r^2}$.

Let $\triangle \tilde{p}\tilde{x}\tilde{y}$ be the corresponding geodesic triangle in the model space. Following the argument of Lemma \ref{hyperbolic}, we have $\angle(\tilde{p}\tilde{x},\tilde{p}\tilde{y}) < \dfrac{d}{(s-d)^\alpha}$. Lemma \ref{angle} follows immediately from the Toponogov comparison theorem. 

Given $\varphi \in C^\infty(S_p)$, we extend it to $M \setminus \{p\}$ by defining $\varphi(r,\theta)=\varphi(\theta)$ for $r>0$. Let $d(x)=r^{-\beta}(x)$. By Lemma \ref{angle} we have $$\mathrm{osc}_{B_x(d)} \varphi \leq C \angle(px,py) = O(r^{-\alpha-\beta}).$$

Let $\chi \in C_0^\infty(\mathbb{R})$ be a function satisfying $0\leq \chi \leq 1$, $\chi(t)=0$ for $|t|\geq1$ and $\chi(t)=1$ for $|t| \leq 1/4$. We define the average $\bar{\varphi}$ of $\varphi$ in the ball $B_x(d)$ as $$\overline{\varphi}(x)=\dfrac{\int_M u(x,y) \varphi(y) dy}{\int_M u(x,y) dy},$$ where $u(x,y)=\chi(r^{2\beta}(x) \rho^2_x(y))$. We have\begin{eqnarray*}
|\overline{\varphi}(x)-\varphi(x)| &=& \frac{\int_{B_x(d(x))} u(x,y) (\varphi(y)-\varphi(x)) dy}{\int_{B_x(d(x))} u(x,y) dy} \\
&\leq &\sup_{y \in B_x(d(x))} |\varphi(y)-\varphi(x)| \\
& =& O(r^{-\alpha-\beta}),
\end{eqnarray*}
which implies $\bar{\varphi}$ and $\varphi$ have the same value on $S(\infty)$. 

Let $$v(x)=\int_M u(x,y) dy,$$ 
we have $\mathrm{Vol}(B_x(d/2) )\leq v(x) \leq \mathrm{Vol}\left(B_x (d)\right)$.

To estimate $\Delta \bar{\varphi}$, we need the following corollary of the Hessian comparison theorem. We omit the proof here because the argument is similar to that of Corollary \ref{Laplacian}.
\begin{corollary}
\label{Laplacian2}
Let $M$ be as in Theorem \ref{main2}. Then we have
$$\dfrac{(n-1)\alpha}{r} \leq \Delta r \leq (n-1)(r^\beta+\dfrac{1}{r}).$$ \end{corollary}

Now we have
\begin{eqnarray}
\label{nabla u 2}
\nabla u &=& \chi'(r^{2\beta}\rho^2) \cdot (2\beta r^{2\beta-1} \rho^2 \nabla r +2r^{2\beta} \rho \nabla \rho) \\ 
&=& \nonumber \mathrm{O}(r^\beta),
\end{eqnarray}
and by Corollary \ref{Laplacian2} that
\begin{multline}
\label{laplacian u 2}
\Delta u = \chi''(r^{2\beta}\rho^2) \cdot (2\beta r^{2\beta-1} \rho^2 \nabla r +2r^{2\beta} \rho \nabla \rho)^2 + \chi'(r^{2\beta}\rho^2) \cdot (2\beta(2\beta-1)r^{2\beta-2}\rho^2 |\nabla r|^2\\
+2 \beta r^{2 \beta-1} \rho^2 \Delta r+8\beta r^{2\beta-1}\rho \nabla \rho \cdot \nabla r+2r^{2\beta}|\nabla \rho|^2+2r^{2\beta} \rho \Delta \rho) = \mathrm{O}(r^{2\beta}). \;\;\;\;\;\;\;\;\;\;\;\;\;\;\;\;\;\;
\end{multline}

Combining \eqref{u/v}, \eqref{nabla u 2} and \eqref{laplacian u 2} we obtain the following estimate:
\begin{eqnarray*}
|\Delta \overline{\varphi}(x)|&=&| \int_M \Delta (\frac{u}{v})(\varphi(y)-\varphi(x))  dy|\\
&\leq & \int_{B_x(d)} |\Delta(\frac{u}{v})| dy \cdot    \mathrm{osc}_{B_x(d)} \varphi \\
&\leq & \sup_{B_x(d)}\{|\Delta(\frac{u}{v})|\}\cdot \mathrm{Vol}(B_x(d(x))) \cdot    \mathrm{osc}_{B_x(d)} \varphi \\
& \leq &\sup_{B_x(d)} \{\frac{|\Delta u|}{v}+2\frac{|\nabla u \cdot \nabla v|}{v^2} +\frac{u |\Delta v|}{v^2}+\frac{2u}{v^3}|\nabla v|^2\} \cdot \mathrm{Vol}(B_x(d(x))) \cdot    \mathrm{osc}_{B_x(d)} \varphi \\
&=& \mathrm{O} \left(\frac{\mathrm{Vol}(B_x)}{\mathrm{Vol}\left(B_x\left(d/2\right)\right)}+\frac{(\mathrm{Vol}(B_x))^2}{\left(\mathrm{Vol}\left(B_x\left(d/2\right)\right)\right)^2}+\frac{(\mathrm{Vol}(B_x))^3}{(\mathrm{Vol}(B_x(d/2)))^3}\right) \cdot \mathrm{O}(r^{2\beta}) \cdot \mathrm{O}(r^{-\alpha-\beta}) \\
&=& \mathrm{O}(r^{-\alpha+\beta}).
\end{eqnarray*}

Consider the function $r^{-\delta}(x)$, where $\delta$ is a positive constant to be chosen later. By Corollary \ref{Laplacian2} We have
$$\Delta r^{-\delta}= \delta r^{-\delta-2}(\delta+1-r\Delta r) \leq \delta r^{-\delta-2}(\delta+1-(n-1)\alpha).$$

Since $\alpha-2>\beta$ and $\alpha>2$, we can choose $\delta<\alpha-2-\beta$ sufficiently small so that $\delta+1-(n-1)\alpha$ is negative. Then we have $$\Delta \overline{\varphi}=\mathrm{O}(r^{-\alpha+\beta})=\mathrm{o}(r^{-\delta-2})=\mathrm{o}(\Delta r^{-\delta}),$$ therefore there exists a constant $c>0$ such that $$\Delta(cr^{-\delta}) \leq -|\Delta \overline{\varphi}|.$$
It follows from the Perron's method again that there exists a harmonic function $f$ satisfying
$$\overline{\varphi}-cr^{-\delta} \leq f \leq \overline{\varphi}+cr^{-\delta}.$$
In particular, $f=\varphi$ on $S(\infty)$ and Theorem \ref{main2} follows.\\

\section{Proof of Theorem \ref{Martin2}}

Throughout this section $M$ is a complete, simply connected Riemannian manifold satisfying $K_M \leq -\dfrac{\alpha(\alpha-1)}{r^2}$ for some $\alpha>2$ and $Ric_M \geq -r^{2\beta}$ for some $\beta<\dfrac{\alpha-4}{3}$.

Given $\omega \in S_p$. Denote $p'=\mathrm{exp}_p \omega$. Let $0<\theta_2<\theta_1<\pi/4$ and $\theta_3=(\theta_1+\theta_2)/2$.

The proof of Theorem \ref{Martin2} is very similar to that of Theorem \ref{Martin}. As remarked in Section 4, it is sufficient to show the following two boundary Harnack inequalities.

\begin{theorem}
\label{Harnack 2.1}
Let $u$ be a positive harmonic function on $C_p(\theta_1)$ which is continuous on  $\overline{C_p(\theta_1)}$ and  vanishes on $\overline{C_p(\theta_1)} \cap S(\infty)$. Then for all $x \in T(\theta_2,1)$,
$$u(x) \leq \tilde{C}r^{-\eta}(x) u(p'),$$
where $\tilde{C}$ and $\eta$ depend only on $n$, $\alpha$, $\beta$, $\theta_1$ and $\theta_2$.
\end{theorem}

\begin{theorem}
\label{Harnack 2.2}
Let $u$, $v$ be two positive harmonic functions on  $C_p(\theta_1)$ which are continuous on  $\overline{C_p(\theta_1)}$ and which vanish on $\overline{C_p(\theta_1)} \cap S(\infty)$. Then for all $x \in T(\theta_2,1)$,
\begin{displaymath}
 \tilde{C}^{-1}\frac{u(p')}{v(p')} \leq \frac{u(x)}{v(x)} \leq \tilde{C}\frac{u(p')}{v(p')},
\end{displaymath}
where $\tilde{C}$ depend only on $n$, $\alpha$, $\beta$, $\theta_1$ and $\theta_2$.
\end{theorem}

Under the curvature condition
$$K_M \leq -\dfrac{\alpha(\alpha-1)}{r^2}, \;\;\;\; Ric_M \geq -r^{2\beta},$$
there exists a cut-off function $\varphi \in C^\infty(M)$ and a constant $R_0>0$ such that \begin{equation}
\label{cut-off 2}
\left \{
\begin{array}{rcl}
\varphi &=& a  \text{ on } T(\theta_2,R_0),\\
\varphi &=& b \text{ on } \partial C_p(\theta_1)\setminus B_p(R_0),\\
|\nabla \varphi| &=& \mathrm{O}(r^{-\alpha}) \text{ on } T(\theta_2,R_0),\\
|\Delta \varphi| &=& \mathrm{O}(r^{\beta-\alpha}) \text{ on } T(\theta_2,R_0)
\end{array} \right.
\end{equation}
for given $a, b$.

\begin{lemma}
\label{weak harnack 2}
Let $u$ be a positive harmonic function on $C_p(\theta_3)$ which is continuous on  $\overline{C_p(\theta_3)}$ and which vanishes on $\overline{C_p(\theta_3)} \cap S(\infty)$. Then for all $x \in T(\theta_2,1)$,
\begin{displaymath}
u(x) \leq \tilde{C} r^{-\eta}(x) \sup_{\partial C_p(\theta_3)} u,
\end{displaymath}
where $\tilde{C}$ and $\eta$ are constants depending only on $n$, $\alpha$, $\beta$,c $\theta_1$ and $\theta_2$.
\end{lemma}
\begin{proof}
By \eqref{cut-off 2}, there exists $\varphi \in C^{\infty}(M)$ and a constant $R_0>0$ such that
\begin{equation*}
\left \{
\begin{array}{rcl}
\varphi &=& 0  \text{ on } T(\theta_2,R_0),\\
\varphi &=& 1 \text{ on } \partial C_p(\theta_3)\setminus B_p(R_0),\\
|\nabla \varphi| &=& \mathrm{O}(r^{-\alpha}) \text{ on } T(\theta_3,R_0),\\
|\Delta \varphi| &=& \mathrm{O}(r^{\beta-\alpha}) \text{ on } T(\theta_3,R_0).
\end{array} \right.
\end{equation*}

Consider $r^{-\eta}$, where $\eta<\alpha-2-\beta$ is a sufficiently small positive number. We have 
\begin{eqnarray*}
\Delta r^{-\eta} &=&\eta r^{-\eta-2}(\eta+1-r\Delta r)\\
&\leq& \eta r^{-\eta-2}(\eta+1-(n-1)\alpha)<0.
\end{eqnarray*}

Since $\Delta \varphi=\mathrm{O}(r^{\beta-\alpha})=\mathrm{o}(r^{-\eta-2})$, we have $|\Delta \varphi| \leq C_1 \Delta r^{-\eta}$ on $T(\theta_3,R_0)$ for $C_1$ and $R_0$ sufficiently large.

Let $f=\varphi +\tilde{C}r^{-\eta}$. We have $\Delta f \leq 0$ on $T(\theta_3,R_0)$ if $\tilde{C}>C_1$. Also $f \geq 1$ on $\partial T(\theta_3,R_0)$, provided $\tilde{C}$ is sufficiently large.

Now consider $\bar{u}=u / \sup_{\partial C_p(\theta_3)} u$. $\bar{u}$ is harmonic and $\bar{u} \leq 1$ on $\partial C_p(\theta_3)$. We have $\Delta(\bar{u}-f)\geq 0$ on $T(\theta_3,R_0)$ and $\bar{u}-f \leq 0$ on $\partial T(\theta_3,R_0)$. By the maximum principle, $\bar{u} \leq f$ on $T(\theta_3,R_0)$. In particular,
$$u(x) \leq f(x) \sup_{\partial C_p(\theta_3)} u=\tilde{C}r^{-\eta} \sup_{\partial C_p(\theta_3)} u$$
for all $x \in T(\theta_2,R_0)$.
\end{proof}

To prove Theorem \ref{Harnack 2.1}, it is sufficient to show that harmonic functions satisfying the given conditions and $u(p')=1$ are uniformly bounded on $\partial C_p(\theta_3)$.

Let $\delta=(\dfrac{\alpha-4}{3}-\beta)/2>0$. Applying Theorem \ref{gradient estimate} we get the following corollary.
\begin{corollary}
Let $u$ be a positive harmonic function on $M$, then
\begin{displaymath}
|\dfrac{\nabla u}{u}|(x) \leq \tilde{C}r^{\beta+\delta},
\end{displaymath}
where $\tilde{C}$ depends only on $n$. 
\end{corollary}

Therefore
\begin{eqnarray}
\label{log u}|\log u(x)|&=&|\log u(x) -\log u(p')|\\
\nonumber&\leq& \sup_{B_p(r(x))} |\nabla \log u| \cdot  d(x,p')\\
\nonumber&=& \mathrm{O}(r^{\beta+\delta+1}).
\end{eqnarray}

By \eqref{cut-off 2}, there exists $\varphi \in C^{\infty}(M)$ with $\dfrac{2}{3} \leq \varphi \leq 1$ and a constant $R_0>0$ such that
\begin{equation*}
\left \{
\begin{array}{rcl}
\varphi &=& \dfrac{2}{3}  \text{ on } T(\theta_3,R_0),\\
\varphi &=& 1 \text{ on } \partial C_p(\theta_1)\setminus B_p(R_0),\\
|\nabla \varphi| &=& \mathrm{O}(r^{-\alpha}) \text{ on } T(\theta_1,R_0),\\
|\Delta \varphi| &=& \mathrm{O}(r^{\beta-\alpha}) \text{ on } T(\theta_1,R_0).
\end{array} \right.
\end{equation*}

Consider the function $u^\varphi$. We have
\begin{eqnarray}
\nonumber \Delta u^\varphi&=&u^\varphi(|\log u \nabla \varphi+\varphi \nabla \log u|^2+\log u \Delta \varphi+2 \nabla \varphi \cdot \nabla \log u-\varphi |\nabla \log u|^2) \\
&=& u^\varphi (\varphi^2-\varphi)|\nabla \log u|^2+u^\varphi \cdot \mathrm{O}(r^{-\alpha+2\beta+2\delta+1}).\label{Delta 2}
\end{eqnarray}

Let $\psi \in C^\infty(\mathbb{R})$ be a function such that
\begin{equation*}
\left \{
\begin{array}{rl}
&1 \leq \psi \leq C_1,\\
&\psi'(t)=-\dfrac{1}{|t|\log^2 |t|} \text{ for } |t| \geq R_1,\\
&-\psi'(t) \geq 12|\psi''(t)| \geq 0 \text{ for } \text{ for all } t.
\end{array} \right.
\end{equation*}

Let $\xi=\log u$. Set $$F(x)=\psi(\xi(x)-r^{-\varepsilon}(x)) \cdot u^\varphi,$$ where $\varepsilon $ is a positive number to be determined later. For $\psi=\psi(\xi(x)-r^{-\varepsilon}(x))$,
$$\nabla \psi=\psi' \cdot (\nabla \xi+\varepsilon r^{-1-\varepsilon} \nabla r)  ,$$
\begin{eqnarray*}
\Delta \psi&=&\psi'' \cdot |\nabla \xi+\varepsilon r^{-1-\varepsilon} \nabla r|^2+\psi'\cdot (-|\nabla \xi|^2+\varepsilon  r^{-2-\varepsilon}(r\Delta r-(1+\varepsilon)).
\end{eqnarray*}

Therefore
\begin{eqnarray*}
\Delta F &=&\psi \Delta u^\varphi +\Delta \psi u^\varphi  +2\nabla \psi \cdot \nabla u^\varphi\\
&=& \psi \Delta u^\varphi +\psi''  u^\varphi \cdot |\nabla \xi+\epsilon r^{-1-\epsilon} \nabla r|^2+\psi'u^\varphi \cdot ((2\varphi-1)|\nabla \xi|^2\\
&&+2\xi \nabla \varphi \cdot \nabla \xi+2\varepsilon r^{-1-\varepsilon}\varphi \nabla \xi \cdot \nabla r+2\varepsilon r^{-1-\varepsilon} \xi \nabla \varphi \cdot \nabla r+\varepsilon r^{-2-\varepsilon}(r \Delta r -1-\varepsilon))\\
&\leq& u^\varphi [(\dfrac{\varphi'}{3}+2\varphi'')|\nabla \xi|^2+C_2 r^{-\alpha+2\beta+2\delta+1}+\varepsilon \varphi'r^{-2-\varepsilon}(r\Delta r-1-\varepsilon)].\\
\end{eqnarray*}

Since $r\Delta r \geq (n-1)\alpha$, we can take $\varepsilon < \max\{(n-1)\alpha-1, \delta\}$ suffciently small so that 
\begin{eqnarray*}
\Delta F &\leq & u^\varphi (C_2 r^{-\alpha+2\beta+2\delta+1}-C_3 \dfrac{r^{-2-\varepsilon}}{r^{\beta+\delta+1} \log^2 r})\\
&\leq & 0
\end{eqnarray*}
at points $x$ such that $r(x) \geq R_2$ and $\xi(x)-r^{-\varepsilon}(x) \geq R_1$ .

When $\xi-r^{-\varepsilon} \leq R_1$, we have that $u=e^{\xi}$ is bounded on $T(\theta_1,R_2)$. In addition, $u^\varphi \leq C_4 u^{1/2}$ and $u^\varphi |\xi| \leq C_5 u^{1/2}$. Using the fact that $|\psi|$, $|\psi'|$ and $|\psi''|$ are all bounded, we conclude that
\begin{displaymath}
\Delta F \leq C_6 u^{1/2}|\nabla \xi|^2 +C_7 r^{-2-\varepsilon}.
\end{displaymath}

Define $$G=F+C_8 u^{1/2}+r^{-\varepsilon_0},$$ we have $\Delta G \leq 0$ on $T(\theta_1,R_3)$ for $0<\varepsilon_0<\varepsilon$ and $C_8$, $R_3$ sufficiently large. 

Since $u$ is harmonic,
$$\Delta (G-u) \leq 0$$ on $T(\theta_1,R_3)$ and $$G-u\geq F-u\geq \psi u^\varphi-u \geq 0$$ on $\partial T(\theta_1,R_3)$. By the maximum principle, 
$$u \leq G$$ on $T(\theta_1,R_3)$. In particular, $$u \leq G \leq  C_1 u^{2/3}+C_8 u^{1/2}+r^{-\varepsilon_0}$$ on $T(\theta_3, R_3)$, which implies that $u$ is uniformly bounded on $T(\theta_3, R_3)$. By Lemma \ref{weak harnack 2} and the gradient estimate we have for all $x \in T(\theta_2,1)$,
\begin{displaymath}
u(x) \leq \tilde{C} r^{-\eta}(x)u(p'),
\end{displaymath}
which completes the proof of Theorem \ref{Harnack 2.1}.

\begin{proof}[Proof of Theorem \ref{Harnack 2.2}]

Without loss of generality we may assume that $u(p')=v(p')=1$.

By Theorem \ref{Harnack 2.1},
\begin{equation}
\label{u,v bound} u, v \leq C_1 r^{-\eta}
\end{equation}
on $T(\theta_3, 1)$. Let $\xi=-\log u$. It follows from \eqref{log u} that
\begin{equation}
\label{u,v 2}
\xi^{-\epsilon} \geq C_2 r^{-\epsilon(\delta+\beta+1)} \geq C_3 v
\end{equation}
on $T(\theta_3, R_1)$ for $\epsilon>0$ sufficiently small.

By \eqref{cut-off 2}, there exists $\varphi \in C^{\infty}(M)$ with $0\leq \varphi \leq 1$ such that
\begin{equation*}
\left \{
\begin{array}{rcl}
\varphi &=& 0  \text{ on } T(\theta_2,R_1),\\
\varphi &=& 1 \text{ on } \partial C_p(\theta_3)\setminus B_p(R_1),\\
|\nabla \varphi| &=& \mathrm{O}(r^{-\alpha}) \text{ on } T(\theta_3,R_1),\\
|\Delta \varphi| &=& \mathrm{O}(r^{\beta-\alpha}) \text{ on } T(\theta_3,R_1).
\end{array} \right.
\end{equation*}

Consider the function $f=u^{1-\varphi} \xi^{-\epsilon \varphi}$. We have
\begin{equation}
\nabla f=f\cdot(\xi \nabla \varphi-(1-\varphi)\nabla \xi -\epsilon \log \xi \nabla \varphi -\epsilon \varphi \nabla \log\xi),
\end{equation}
and
\begin{eqnarray*}
\Delta f&=&f \cdot (|\xi \nabla \varphi-(1-\varphi)\nabla \xi -\epsilon \log \xi \nabla \varphi -\epsilon \varphi \nabla \log\xi|^2\\
&&+\xi \Delta \varphi +2\nabla \varphi \cdot \nabla \xi-(1-\varphi)\Delta \xi-\epsilon \log \xi \Delta \varphi-2\epsilon \nabla \log \xi \cdot \nabla \varphi-\epsilon \varphi \Delta \log \xi)\\
&\leq& f \cdot [(\varphi^2-\varphi) |\nabla \xi|^2+\epsilon \varphi(1-2\varphi)\dfrac{|\nabla \xi|^2}{\xi}+(\epsilon^2 \varphi^2+\epsilon \varphi)|\nabla \log \xi|^2+C_4 r^{-(\alpha-2\beta-1-\delta)}].
\end{eqnarray*}

Let $\psi \in C^\infty(\mathbb{R})$ be a function such that
\begin{equation*}
\left \{
\begin{array}{rl}
&1 \leq \psi \leq C_5,\\
&\psi'(t)=\dfrac{1}{|t|\log^2 |t|} \text{ for } |t| \geq R_2,\\
&\psi'(t) \geq 12|\psi''(t)| \geq 0 \text{ for } \text{ for all } t.
\end{array} \right.
\end{equation*}

Set
\begin{equation}
F(x)=\psi(\xi+r^{-\varepsilon}) \cdot f,
\end{equation}
where $\varepsilon$ is a positive number to be determined later. We have for $\psi=\psi(\xi+r^{-\varepsilon})$,
$$\nabla \psi=\psi' \cdot (\nabla \xi-\varepsilon r^{-1-\varepsilon} \nabla r),$$
\begin{eqnarray*}
\Delta \psi&=&\psi'' \cdot |\nabla \xi-\varepsilon r^{-1-\varepsilon} \nabla r|^2+\psi'\cdot (\Delta \xi+\varepsilon r^{-2-\varepsilon}(\varepsilon+1-r\Delta r)).
\end{eqnarray*}

Then we have
\begin{eqnarray*}
\Delta F &=&\psi \Delta f +\Delta \psi f  +2\nabla \psi \cdot \nabla f\\
&\leq & \psi f[\varphi(\varphi-1)|\nabla \xi|^2+\varphi(\varepsilon-2\varepsilon\varphi)\dfrac{|\nabla \xi|^2}{\xi}+\varphi(\varepsilon^2\varphi+\varepsilon)\dfrac{|\nabla \xi|^2}{\xi^2}\\
&&+C_4 r^{-(\alpha-2\beta-\delta-1)}]+\psi' f[(2\varphi-1)|\nabla \xi|^2-2\epsilon \varphi \dfrac{|\nabla \xi|^2}{\xi}+2\varepsilon \epsilon \varphi r^{-1-\varepsilon}|\nabla \log \xi|\\
&&+C_6 r^{-(\alpha-2\beta-2\delta-1)}+\varepsilon r^{-2-\varepsilon}(\varepsilon+1-r\Delta r)]\\
&\leq& f[C_7 r^{-(\alpha-2\beta-2\delta-1)}+C_8 \varepsilon r^{-\beta-2\delta-1}r^{-2-\varepsilon}(\varepsilon+1-r\Delta r)]
\end{eqnarray*}
on $T(\theta_3,R_2)$.

We can choose $\varepsilon<(\alpha-2\beta-2\delta-1)-(\beta+2\delta+3)=2\delta$ to be sufficiently small so that
$$\Delta F \leq 0$$
on $T(\theta_3,R_3)$ for $R_3>\max\{R_0,R_1,R_2\}$ sufficiently large.

We alreday know from \eqref{u,v 2} that
$$F=\psi \xi^{-\epsilon} \geq \xi^{-\epsilon} \geq C_3v$$
on $\partial C_p(\theta_3)\setminus B_p(R_3)$. Since $v$ is harmonic, it follows from the gradient estimate that
$$F \geq C_9 v$$ on $\partial T(\theta_3,R_3)$.

On the other hand, on $T(\theta_3,R_3)$ we have 
$$\Delta(F-C_9 v)=\Delta F \leq 0.$$

By the maximum principle, we have $$F\geq C_9v$$ on $\overline{T(\theta_3,R_3)}$. In particular, $$v \leq \dfrac{1}{C_9} F \leq \dfrac{C_5}{C_9} u$$ on $T(\theta_2, R_3)$. By the Harnack inequality $\dfrac{u}{v}$ is uniformly bounded on $T(\theta_2,1)$.  
\end{proof}

\bibliographystyle{amsalpha}
\bibliography{Dirichlet}

\end{document}